\documentclass[preprint,review,10pt]{amsart}
\usepackage{amsmath,amssymb,amsfonts,xspace}
\newtheorem{theorem}{Theorem}[section]

\theoremstyle{definition}
\newtheorem{definition}[theorem]{Definition}
\newtheorem{example}[theorem]{Example}

\newtheorem{proposition}[theorem]{Proposition}
\newtheorem{corollary}[theorem]{Corollary}
\newtheorem{lem}[theorem]{Lemma}

\theoremstyle{remark}

\numberwithin{equation}{section}

\begin{document}

\title{  Weakly $S$-prime hyperideals }

\author{Mahdi Anbarloei}
\address{Department of Mathematics, Faculty of Sciences,
Imam Khomeini International University, Qazvin, Iran.
}

\email{m.anbarloei@sci.ikiu.ac.ir}





\begin{abstract}
 In this paper, we aim to introduce weakly $n$-ary $S$-prime hyperideals in a commutative Krasner $(m,n)$-hyperring.

\end{abstract}
\maketitle

\section{Introduction} 
In recent years, prime ideals and their expansions have gained significant attention in commutative algebra, drawing the attention of numerous authors. One of these extensions called  $S$-prime ideals was introduced via a multiplicative subset of a  commutative ring in \cite{hamed}. Afterward, the idea of weakly $S$-prime ideals was investigated by Almahdi et al. in \cite{Almahdi}. This study demonstrated that weakly $S$-prime ideals exhibits numerous properties similar to those of weakly prime ideals. Let $A$ be a commutative ring,  $Q$ a ideal of $A$ and $S \subseteq A$  a multiplicative set. The  ideal $Q$  is said to be weakly $S$-prime if there exists $s \in S$ such that, $0 \neq xy \in Q$ for $x,y \in A$ implies $sx \in Q$ or $sy \in Q$.  

Hyperstructures represent a natural extension of classical algebraic structures and they were
defined by the French mathematician F. Marty. In 1934, Marty \cite{s1} defined the concept of a hypergroup as a generalization of groups during the $8^{th}$ Congress of the Scandinavian Mathematicians. A comprehensive
review of the theory of hyperstructures can be found in \cite {amer2, s2, s3, davvaz1, davvaz2, s4, s10, jian}. The simplest algebraic hyperstructures which possess the properties of closure and associativity are called semihypergroups. 
$n$-ary semigroups and $n$-ary groups are algebras with one $n$-ary operation which is associative and invertible in a generalized sense. The notion of investigations of $n$-ary algebras goes back to Kasner’s lecture \cite{s5} at a scientific meeting in 1904. In 1928, Dorente wrote the first paper concerning the theory of $n$-ary groups \cite{s6}. Later on, Crombez and
Timm \cite{s7, s8} defined and described the notion of the $(m, n)$-rings and their quotient structures. Mirvakili and Davvaz [20] defined $(m,n)$-hyperrings and obtained several results in this respect. In \cite{s9}, they introduced and illustrated a generalization of the notion of a hypergroup in the sense of Marty and a generalization of an $n$-ary group, which is called $n$-ary hypergroup.
The $n$-ary structures has been studied in \cite{l1, l2, l3, ma, rev1}. Mirvakili and Davvaz \cite{cons} defined $(m,n)$-hyperrings and obtained several results in this respect.

It was Krasner, who introduced one important class of hyperrings, where the addition is a hyperoperation, while the multiplication is an ordinary binary operation, which is called Krasner hyperring. In \cite{d1}, a generalization of the Krasner hyperrings, which is a subclass of $(m,n)$-hyperrings, was defined by Mirvakili and Davvaz. It is called Krasner $(m,n)$-hyperring. Ameri and Norouzi in \cite{sorc1} introduced some important
hyperideals such as Jacobson radical, n-ary prime and primary hyperideals, nilradical, and n-ary multiplicative subsets of Krasner $(m, n)$-hyperrings. Afterward, the notion of $(k,n)$-absorbing hyperideal  were studied by Hila et al. \cite{rev2}. The idea was extended to weakly $(k,n)$-absorbing hyperideals by Davvaz et al. in \cite{davvazz}.
The notion of $S$-prime hyperideals  which is an extension of the prime hyperideals via an $n$-ary multiplicative subset of a Krasner $(m,n)$-hyperring was investigated in \cite{mah6}. \\

In this paper, motivated by the research work on weakly $S$-prime ideals, the notion of weakly $n$-ary $S$-prime hyperideals is introduced and investigated in a commutative Krasner $(m,n)$-hyperring $A$ where $S \subseteq A$ is a multiplicative set. Among many results in this paper, we prove that every  $n$-ary prime hyperideals of $A$ is a weakly $n$-ary $S$-prime hyperideal but the converse may not be always true   in Example \ref{madar}. It is shown that every weakly $n$-ary $S$-prime hyperideal of $A$ is $n$-ary prime if and only if $A$ is an  $n$-ary hyperintegral domain and every $S$-prime hyperideal of $A$ is $n$-ary prime in Theorem \ref{13}. We obtain that if $Q$ is a strongly weakly $n$-ary $S$-prime hyperideal of $A$ that is not $n$-ary $S$-prime, then $g(Q^{(n)})=0$ in Theorem \ref{14}. In theorem \ref{naka}, we propose a type of Nakayama$^,$s Lemma for strongly weakly $n$-ary $S$-prime hyperideals. Moreover, we conclude that if $Q$ is a strongly weakly $n$-ary $S$-prime hyperideal of $A$ that is not $n$-ary $S$-prime, then $g(g(s,rad(0),1_A^{(n-2)}),Q^{(n-1)})=0$ for some $s \in S$ in Theorem \ref{18}.
Finally, we present some characterizations of this notion on cartesian product of commutative Krasner $(m,n)$-hyperring. 
\section{Preliminaries}
Recall first the definitions and basic terms from the hyperrings theory. \\ 
Let $A$ be a non-empty set and $P^*(A)$ be the
set of all the non-empty subsets of $H$. Then the mapping $f : H^n \longrightarrow P^*(A)$
is called an $n$-ary hyperoperation and the algebraic system $(A, f)$ is called an $n$-ary hypergroupoid. For non-empty subsets $U_1,..., U_n$ of $H$, define
\[f(U^n_1) = f(U_1,..., U_n) = \bigcup \{f(x^n_1) \ \vert \ a_i \in U_i, i = 1,..., n \}.\]
The sequence $x_i, x_{i+1},..., x_j$ 
will be denoted by $x^j_i$. For $j< i$, $x^j_i$ is the empty symbol. Using this notation,
\[f(x_1,..., x_i, y_{i+1},..., y_j, z_{j+1},..., z_n)\]
will be written as $f(x^i_1, y^j_{i+1}, z^n_{j+1})$. The expression will be written in the form \[f(x^i_1, y^{(j-i)}, z^n_{j+1}),\] when $y_{i+1} =... = y_j = y$. 
If for every $1 \leq i < j \leq n$ and all $x_1, x_2,..., x_{2n-1} \in A$, 
\[f\bigg(x^{i-1}_1, f(x_i^{n+i-1}), x^{2n-1}_{n+i}\bigg) = f\bigg(x^{j-1}_1, f(x_j^{n+j-1}), x_{n+j}^{2n-1}\bigg)\]
then the n-ary hyperoperation $f$ is called associative. An $n$-ary hypergroupoid with the
associative $n$-ary hyperoperation is called an $n$-ary semihypergroup. 

An $n$-ary hypergroupoid $(A, f)$ in which the equation $b \in f(a_1^{i-1}, x_i, a_{ i+1}^n)$ has a solution $x_i \in A$
for every $a_1^{i-1}, a_{ i+1}^n,b \in A$ and $1 \leq i \leq n$, is called an $n$-ary quasihypergroup, when $(A, f)$ is an $n$-ary
semihypergroup, $(A, f)$ is called an $n$-ary hypergroup. 

An $n$-ary hypergroupoid $(A, f)$ is commutative if for all $ \sigma \in \mathbb{S}_n$, the group of all permutations of $\{1, 2, 3,..., n\}$, and for every $a_1^n \in A$ we have $f(a_1,..., a_n) = f(a_{\sigma(1)},..., a_{\sigma(n)})$.
If $a_1^n \in A$, then $(a_{\sigma(1)},..., a_{\sigma(n)})$ is denoted by $a_{\sigma(1)}^{\sigma(n)}$.

If $f$ is an $n$-ary hyperoperation and $t = l(n- 1) + 1$, then $t$-ary hyperoperation $f_{(l)}$ is given by

\[f_{(l)}(x_1^{l(n-1)+1}) = f\bigg(f(..., f(f(x^n _1), x_{n+1}^{2n -1}),...), x_{(l-1)(n-1)+1}^{l(n-1)+1}\bigg).\]
\begin{definition}
\cite{d1} A non-empty subset $K$ of an $n$-ary hypergroup $(A, f)$ is said to be 
an $n$-ary subhypergroup of $A$ if $(K,f)$ is an $n$-ary hypergroup.
An element $e \in A$ is called a scalar neutral element if $x = f(e^{(i-1)}, x, e^{(n-i)})$, for every $1 \leq i \leq n$ and
for every $x \in A$. 

An element $0$ of an $n$-ary semihypergroup $(A, g)$ is called a zero element if for every $x^n_2 \in A$, 
$g(0, x^n _2) = g(x_2, 0, x^n_ 3) = ... = g(x^n _2, 0) = 0$.
If $0$ and $0^ \prime $ are two zero elements, then $0 = g(0^ \prime , 0^{(n-1)}) = 0 ^ \prime$ and so the zero element is unique. 
\end{definition}
\begin{definition}
\cite{l1} Let $(A, f)$ be a $n$-ary hypergroup. $(A, f)$ is called a canonical $n$-ary
hypergroup if:\\
(1) There exists a unique $e \in A$, such that for every $x \in A, f(x, e^{(n-1)}) = x$;\\
(2) For all $x \in A$ there exists a unique $x^{-1} \in A$, such that $e \in f(x, x^{-1}, e^{(n-2)})$;\\
(3) If $x \in f(x^n _1)$, then for all $i$, $x_i \in f(x, x^{-1},..., x^{-1}_{ i-1}, x^{-1}_ {i+1},..., x^{-1}_ n)$.

$e$ is said to be the scalar identity of $(A, f)$ and $x^{-1}$ is the inverse of $x$. Notice that the inverse of $e$ is $e$.
\end{definition}
\begin{definition}
\cite{d1} $(A, f, g)$, or simply $A$, is said to be a Krasner $(m, n)$-hyperring if:\\
(1) $(A, f$) is a canonical $m$-ary hypergroup;\\
(2) $(A, g)$ is a $n$-ary semigroup;\\
(3) The $n$-ary operation $g$ is distributive with respect to the $m$-ary hyperoperation $f$ , i.e., for every $a^{i-1}_1 , a^n_{ i+1}, x^m_ 1 \in A$, and $1 \leq i \leq n$,
\[g\bigg(a^{i-1}_1, f(x^m _1 ), a^n _{i+1}\bigg) = f\bigg(g(a^{i-1}_1, x_1, a^n_{ i+1}),..., g(a^{i-1}_1, x_m, a^n_{ i+1})\bigg);\]
(4) $0$ is a zero element (absorbing element) of the $n$-ary operation $g$, i.e., for every $x^n_ 2 \in R$ , 
$g(0, x^n _2) = g(x_2, 0, x^n _3) = ... = g(x^n_ 2, 0) = 0$.
\end{definition}
Throughout this paper, $(A, f, g)$ denotes a commutative Krasner $(m,n)$-hyperring.

A non-empty subset $B$ of $A$ is called a subhyperring of $A$ if $(B, f, g)$ is a Krasner $(m, n)$-hyperring. 
The non-empty subset $Q$ of $(A,f,g)$ is a hyperideal if $(Q, f)$ is an $m$-ary subhypergroup
of $(A, f)$ and $g(x^{i-1}_1, Q, x_{i+1}^n) \subseteq Q$, for every $x^n _1 \in A$ and $1 \leq i \leq n$.

\begin{definition} \cite{sorc1} Let $x$ be an element in a Krasner $(m,n)$-hyperring $A$. The hyperideal generated by $x$ is denoted by $<x>$ and defined as follows:
\[<x>=g(A,x,1^{(n-2)})=\{g(r,x,1^{(n-2)}) \ \vert \ r \in A\}.\]
Also, let $Q$ be a hyperideal of $A$. Then define $(Q:x)=\{a \in A \ \vert \ g(a,x,1_A^{(n-2)}) \in Q\}$. 
\end{definition}

\begin{definition} \cite{sorc1}
An element $x \in A$ is said to be invertible if there exists $y \in A$ such that $1_A=g(x,y,1_A^{(n-2)})$. Also,
the subset $U$ of $A$ is invertible if and only if every element of $U$ is invertible.
\end{definition}
\begin{definition}
\cite{sorc1} Let $Q \neq A$ be a hyperideal of a Krasner $(m, n)$-hyperring $A$. $Q$ is a prime hyperideal if for hyperideals $Q_1,..., Q_n$ of $A$, $g(Q_1^ n) \subseteq P$ implies that $Q_1 \subseteq Q$ or $Q_2 \subseteq Q$ or ...or $Q_n \subseteq Q$.
\end{definition}
\begin{lem} 
(Lemma 4.5 in \cite{sorc1}) Let $Q\neq A$ be a hyperideal of a Krasner $(m, n)$-hyperring $A$. Then $Q$ is a prime hyperideal if for all $x^n_ 1 \in A$, $g(x^n_ 1) \in Q$ implies that $x_1 \in Q$ or ... or $x_n \in Q$. 
\end{lem}
\begin{definition} \cite{sorc1} Let $Q$ be a hyperideal in a Krasner $(m, n)$-hyperring $A$ with
scalar identity. The radical (or nilradical) of $Q$, denoted by $rad(Q)$
is the hyperideal $\bigcap Q$, where
the intersection is taken over all prime hyperideals $Q$ which contain $Q$. If the set of all prime hyperideals containing $Q$ is empty, then $rad(Q)$ is defined to be $A$.
\end{definition}
Ameri and Norouzi showed that if $x \in rad(Q)$ then 
there exists $u \in \mathbb {N}$ such that $g(x^ {(u)} , 1_A^{(n-u)} ) \in Q$ for $u \leq n$, or $g_{(l)} (x^ {(u)} ) \in Q$ for $u = l(n-1) + 1$ \cite{sorc1}.
\begin{definition}
\cite{sorc1} Let $Q$ be a proper hyperideal of a Krasner $(m, n)$-hyperring $A$ with the scalar identity $1_A$. $Q$ is a primary hyperideal if $g(x^n _1) \in Q$ and $x_i \notin Q$ implies that $g(x_1^{i-1}, 1_A, x_{ i+1}^n) \in rad(Q)$ for some $1 \leq i \leq n$.
\end{definition}
If $Q$ is a primary hyperideal in a Krasner $(m, n)$-hyperring $A$ with the scalar identity $1_A$, then $rad(Q)$ is prime. (Theorem 4.28 in \cite{sorc1})
\begin{definition} \cite{sorc1}
Let $S \subseteq A$ be a non-empty subset of a Krasner $(m,n)$-hyperring $R$. $S$ is called an n-ary multiplicative if $g(s_1^n) \in S$ for $s_1,...,s_n \in S$.
\end{definition}

\section{weakly $n$-ary S-prime hyperideals}
In a very recent paper \cite{mah6}, the concept of $n$-ary prime hyperideals was generalized to $n$-ary $S$-prime hyperideals via an $n$-ary multiplicative subset  in a Krasner $(m,n)$-hyperring. Assume that $Q$ is a hyperideal of a commutative Krasner $(m,n)$-hyperring $A$ and $S \subseteq A$ is an $n$-ary multiplicative set such that $Q \cap S=\varnothing$. The hyperideal $Q$ is called $n$-ary $S$-prime if there exists an $s \in S$ such that for every $a_1^n \in A$ with $ g(a_1^n) \in Q$,  $g(s,a_i,1_A^{(n-2)}) \in Q$ for some $i \in \{1,\cdots,n\}$.  Our purpose here is  to introduce and study the notion of weakly $n$-ary $S$-prime hyperideals which constitutes a generalization of $n$-ary $S$-prime hyperideals. The weakly-version of $n$-ary $S$-prime hyperideals is defined as follows. 
\begin{definition}
Assume that $Q$ is a hyperideal of  $A$ and $S$ is an $n$-ary multiplicative subset of $A$ such that $Q \cap S=\varnothing$. We say that $Q$ is a weakly $n$-ary $S$-prime hyperideal of $A$ if there exists an $s \in S$ such that for all $a_1^n \in A$ if $0 \neq g(a_1^n) \in Q$, we have $g(s,a_i,1_A^{(n-2)}) \in Q$ for some $i \in \{1,\cdots,n\}$.  
\end{definition}
\begin{example} \label{madar}
The set $A=\{0,1,2,3\}$ with following 2-hyperoperation $``\boxplus"$ is a canonical 2-ary hypergroup

\hspace{1.5cm}
\[\begin{tabular}{c|c} 
$\boxplus$ & $0$ \ \ \ \ \ \ \ $1$ \ \ \ \ \ \ \ $2$ \ \ \ \ \ \ \ $3$
\\ \hline 0 & $0$\ \ \ \ \ \ \ $1$\ \ \ \ \ \ \ \ \ $2$ \ \ \ \ \ \ \ $3$ 
\\ $1$ & $1$ \ \ \ \ \ \ \ $I$ \ \ \ \ \ \ \ $3$ \ \ \ \ \ \ $J$
\\ $2$ & $2$ \ \ \ \ \ \ \ $3$ \ \ \ \ \ \ \ $0$ \ \ \ \ \ \ \ $1$
\\ $3$ & $3$ \ \ \ \ \ \ \ $J$ \ \ \ \ \ \ \ $1$ \ \ \ \ \ \ \ $I$
\end{tabular}\]

where  $I=\{0,1\}$ and $J=\{2,3\}$. Define a 4-ary operation $g$ on $A$ as $g(a_1^4)=2 $ if  $a_1^4 \in J$ or $0$ if otherwise.
In the  Krasner (2,4)-hyperring $(R,\oplus,g)$, the set $S=\{2,3\}$ is  4-ary multiplicative  and $Q=\{0\}$ is a weakly 4-ary $S$-prime hyperideal  but it is not 4-ary prime, because $g(1,1,2,3)=0 \in Q$ while $1,2,3 \notin Q$.
\end{example}
\begin{example} \label{ex}
Consider Krasner $(3, 3)$-hyperring $(A=\{0,1,2\},f,g)$ where   3-ary hyperoperation $f$ and 3-ary operation $g$ are defined as

$f(0,0,0)=0, \ \ \  f(0,0,2)=2, \ \ \ f(0,1,1)=1, \ \ \ f(1,1,1)=1, \ \ \ f(2,2,2)=2,$

$ f(0,0,1)=1,\ \ \ f(0,2,2)=2,\ \ \ f(1,1,2)=f(1,2,2)=f(0,1,2)=A,$

$ g(1,1,1)=1,\ \ \ \ g(1,1,2)=g(1,2,2)=g(2,2,2)=2,$\\
and  $g(0,a_1^2)=0$   for $a_1^2 \in A$.\\
In the hyperring, the set $S=\{1,2\}$ is 3-ary multiplicative and   $Q=\{0,2\}$ is a  weakly 3-ary $S$-prime hyperideal of $A$.
\end{example}
\begin{proposition} \label{11}
Let $S \subseteq  A$ be an $n$-ary multiplicative set, $Q$   a hyperideal of $A$ with $Q \cap S=\varnothing$ and $Q_1^{n-1}$ some hyperideals of $A$ such that $Q_j \cap S \neq \varnothing $ for each $j \in \{1,\cdots,n-1\}$. If $Q$ is a weakly $n$-ary $S$-prime hyperideal of $A$, then $g(Q_1^{n-1},Q)$ is a weakly $n$-ary $S$-prime hyperideal of $A$.
\end{proposition}
\begin{proof}
Assume that $0 \neq g(a_1^n) \in g(Q_1^{n-1},Q)$ for $a_1^n \in A$. Clearly,  $g(Q_1^{n-1},Q) \cap S=\varnothing$. Since $Q$ is a weakly $n$-ary $S$-prime hyperideal of $A$ and $0 \neq g(a_1^n) \in Q $, there exists an $s \in S$ such that $g(s,a_i,1_A^{(n-2)}) \in Q$ for some $i \in \{1,\cdots,n\}$. Take any $s_j \in  Q_j \cap S$ for each $j \in \{1,\cdots,n-1\}$.  So $g(s_1^{n-1},s) \in S$. Therefore we get $g(g(s_1^{n-1},s),a_i,1_A^{n-2})=g(s_1^{n-1},g(s,a_i,1_A^{(n-2)})) \in g(Q_1^{n-1},Q)$. Thus $g(Q_1^{n-1},Q)$ is a weakly $n$-ary $S$-prime hyperideal of $A$.
\end{proof}
\begin{proposition}\label{12}
Let $S \subseteq  A$ be an $n$-ary multiplicative set and $Q$ be   a hyperideal of $A$ with $Q \cap S=\varnothing$.  If  $P$ is a hyperideal of $A$ such that $P \cap S \neq \varnothing $ and $Q$ is a weakly $n$-ary $S$-prime hyperideal of $A$, then $Q \cap P$ is a weakly $n$-ary $S$-prime hyperideal of $A$.
\end{proposition}
\begin{proof}
Let $0 \neq g(a_1^n) \in Q \cap P$ for $a_1^n \in A$. Since $0 \neq g(a_1^n) \in Q $ and $Q$ is a weakly $n$-ary $S$-prime hyperideal of $A$, there exists an $s \in S$ such that $g(s,a_i,1_A^{(n-2)}) \in Q$ for some $i \in \{1,\cdots,n\}$. Take any $t \in P \cap S$. Then we have $g(t^{(n-1)},s) \in S$. Therefore we get the result that $g(g(t^{(n-1)},s),a_i,1_A^{(n-2)})=g(t^{(n-1)},g(s,a_i,1_A^{(n-2)})) \in Q \cap P$ which shows $Q \cap P$ is a weakly $n$-ary $S$-prime hyperideal of $A$.
\end{proof}
\begin{proposition} \label{31}
Let $S \subseteq  A$ be an $n$-ary multiplicative set and $Q$ be   a hyperideal of $A$ with $Q \cap S=\varnothing$. If $(Q : s)$  is a weakly $n$-ary prime hyperideal of $A$  for some $s \in S$, then $Q$ is a weakly $n$-ary $S$-prime hyperideal of $A$.
\end{proposition}
\begin{proof}
 Let $(Q : s)$ be a weakly $n$-ary prime hyperideal of $A$  for some $s \in S$.  Assume that $0 \neq g(a_1^n)\in Q$ for $a_1^n \in A$. Then we have $0 \neq g(a_1^n)\in (Q : s)$ as $Q \subseteq (Q : s)$. Since $(Q : s)$ is a weakly $n$-ary prime hyperideal of $A$, we obtain $a_i \in (Q : s)$ for some $i \in \{1,\cdots,n\}$ which implies that $g(s,a_i,1_A^{(n-2)})\in Q$. Consequently,  $Q$ is a weakly $n$-ary $S$-prime hyperideal of $A$.
\end{proof}
\begin{theorem}\label{13}
Let $S \subseteq  A$ be an $n$-ary multiplicative set. Then every weakly $n$-ary $S$-prime hyperideal of $A$ is $n$-ary prime if and only if $A$ is an  $n$-ary hyperintegral domain and every $S$-prime hyperideal of $A$ is $n$-ary prime.
\end{theorem}
\begin{proof}
($\Longrightarrow$) Let every weakly $n$-ary $S$-prime hyperideal of $A$ is $n$-ary prime. Since $\langle 0 \rangle$ is a weakly $n$-ary $S$-prime hyperideal of $A$, it is $n$-ary prime which implies $A$ is an  $n$-ary hyperintegral domain. \\

($\Longleftarrow$) Let $A$ be an  $n$-ary hyperintegral domain. This means that every weakly $n$-ary $S$-prime hyperideal of $A$ is $n$-ary $S$-prime and so it is prime by the hypothesis.
\end{proof}
\begin{proposition} \label{1pezeshk}
Let $S \subseteq T \subseteq A$ be two $n$-ary multiplicative sets such that for each $t \in T$, there exists $t^\prime \in T$ with $g(t^{(n-1)},t^\prime) \in S$. If $Q$ is a weakly  $n$-ary $T$-prime hyperideal of $A$, then $Q$ is a weakly $n$-ary $S$-prime hyperideal of $A$.
\end{proposition}
\begin{proof}
Assume that $0 \neq g(a_1^n) \in Q$ for some $a_1^n \in A$. Since $Q$ is a weakly $n$-ary $T$-prime hyperideal of $A$,  there exists $t \in T$ such that $g(t,a_i,1^{(n-2)}) \in Q$ for some $i \in \{1,\cdots,n\}$. By the hypothesis, there exists $t^{\prime} \in T$ such that $g(t^{(n-1)},t^\prime) \in S$. Let $s=g(t^{(n-1)},t^{\prime})$. Therefore we get the result that  $g(s,a_i,1_A^{(n-2)})=g(g(t^{(n-1)},t^{\prime}),a_i,1_A^{(n-2)})
=g(g(t^{(n-2)},t^{\prime},1_A),g(t,a_i,1^{(n-2)}),1_A^{(n-2)}) \in Q$, as needed.
\end{proof}

\begin{theorem}
Let $Q$ be a hyperideal of $A$ and $S \subseteq A$ be an $n$-ary multiplicative set such that $Q \cap S=\varnothing$ and $1_A \in S$. Then $Q$ is a weakly $n$-ary $S$-prime hyperideal of $A$ if and only if $Q$ is a weakly $n$-ary $S^{\star}$-prime hyperideal where $S^{\star}=\{a \in A \ \vert \ \frac{a}{1_A} \ \text{is invertible in} \ S^{-1}A \}$.
\end{theorem}
\begin{proof}
($\Longrightarrow$) Since $Q$ is a weakly $n$-ary $S$-prime hyperideal of $A$ and $S^{\star}$ is an $n$-ary multiplicative subset of $A$ containing $S$, we are done.\\ 

($\Longleftarrow$) Assume that $t \in S^{\star}$. This means that $\frac{t}{1_A}$ is invertible in $S^{-1}A$ and so there exists $x \in A$ and $s \in S$ such that $G(\frac{t}{1_A},\frac{x}{s},\frac{1_A}{1_A}^{(n-2)})=\frac{g(t,x,1_A^{(n-2)})}{g(s,1_A^{(n-1)})}=\frac{1_A}{1_A}$.  Then there exists $s^{\prime} \in S$ such that $0 \in g(s^{\prime}, f(g(t,x,1_A^{(n-2)}),-g(s,1_A^{(n-1)}),0^{(m-2)}),1_A^{(n-2)})=f(g(s^{\prime}, t,x,1_A^{(n-3)}),-g(s^{\prime},s,1_A^{(n-2)}),0^{(m-2)})$. 
Since  $g(s^{\prime},s,1_A^{(n-2)}) \in S$, we obtain $g(s^{\prime}, t,x,1_A^{(n-3)}) \in f(g(s^{\prime},s,1_A^{(n-2)}),0^{(m-1)}) \subseteq S$.
Let  $s^{\prime \prime }=g(s^{\prime},x,1_A^{(n-2)})$. Since $G(\frac{g(s^{\prime},x,1_A^{(n-2)})}{1_A},\frac{g(t,1_A^{(n-1)})}{g(s^{\prime},x,t,1_A^{(n-3)})},\frac{1_A}{1_A}^{(n-2)})=\frac{g(s^{\prime},x,t,1_A^{(n-3)})}{g(s^{\prime},x,t,1_A^{(n-3)})}=\frac{1_A}{1_A}$, we obtain $s^{\prime \prime} \in S^{\star}$.  Hence we have  $g(t^{(n-1)},g({s^{\prime \prime}}^{(n-1)},g(t,s^{\prime \prime},1_A^{(n-2)})))=g(g(t,s^{\prime \prime},1_A^{(n-2)})^n) \in S$. Put  $t^{\prime}=g({s^{\prime \prime}}^{(n-1)},g(t,s^{\prime \prime},1_A^{(n-2)}))$. So $t^{\prime} \in S^{\star}$. Since $g(t^{(n-1)},t^{\prime}) \in S$, we conclude that $Q$ is a weakly $n$-ary $S$-prime hyperideal of $A$ by Proposition \ref{1pezeshk}.
\end{proof}

Recall from \cite{davvazz} that a hyperideal $Q$ of $A$  is called strongly weakly $n$-ary prime  if $0 \neq g(Q_1^n) \subseteq Q$ for all hyperideals $Q_1^n$ of $A$  implies that $Q_i \subseteq Q$ for some $i \in \{1,\cdots,n\}$. Assume that  $S \subseteq A$ is an $n$-ary multiplicative set  satisfying  $Q \cap S= \varnothing$.  A hyperideal $Q$ of $A$  refers to a strongly weakly $n$-ary $S$-prime hyperideal if there exists an $s \in S$ such that for each hyperideal  $Q_1^n $ of $ A$ if $0 \neq g(Q_1^n) \subseteq Q$, we have $g(s,Q_i,1_A^{(n-2)}) \subseteq Q$ for some $i \in \{1,\cdots,n\}$.  In this case, it is said that $Q$ is associated to $s$. It is clear that every strongly weakly $n$-ary $S$-prime hyperideal of $A$  is a weakly $n$-ary $S$-prime hyperideal.
\begin{theorem} \label{jalili} 
Assume that  $Q$ is a strongly weakly $n$-ary $S$-prime hyperideal of $A$ such that $Q$ is associated to $s$.  If $g(a_1^n)=0$ for $a_1^n \in A$ but $g(s,a_i,1_A^{(n-2)}) \notin Q$  for all $i \in \{1,\cdots,n\}$, then $g(a_1,\cdots,\widehat{a_{k_1}}, \cdots,\widehat{a_{k_2}},\cdots, \widehat{a_{k_u}}, \cdots, Q^{(v)})=0$ for each $k_1,\cdots,k_v \in \{1,\cdots,n\}$. 
\end{theorem}
\begin{proof}
We use the induction on $v$. Let $v=1$. Assume that  $g(a_1^{i-1},Q,a_{i+1}^n) \neq 0$ for some $i \in \{1,\cdots,n\}$. Therefore  we get $0 \neq g(a_1^{i-1},a,a_{i+1}^n) \in Q$ for some $a \in Q$. So we conclude that $0 \neq g(a_1^{i-1},a,a_{i+1}^n)=f(g(a_1^n),g(a_1^{i-1},a,a_{i+1}^n),0^{(m-2)})=g(a_1^{i-1},f(a,a_i,0^{(m-2)}),a_{i+1}^n) \subseteq Q$. Then we get  $g(s,f(a,a_i,0^{(m-2)}),1_A^{(n-2)})=f(g(s,a,1_A^{(n-2)}),g(s,a_i,1_A^{(n-2)}),0^{(m-2)}) \subseteq Q$ which means $g(s,a_i,1_A^{(n-2)}) \in Q$ or $g(s,a_j,1_A^{(n-2)}) \in Q$ for some $j \in \{1,\cdots,\widehat{i},\cdots, n\}$. It follows that $g(s,a_i,1_A^{(n-2)}) \in Q$  for some $i \in \{1,\cdots,n\}$ which is impossible. Now, suppose that the
claim is true for all positive integers which are  less than $v$. Suppose on the contrary that  $g(a_1,\cdots,\widehat{a_{k_1}}, \cdots,\widehat{a_{k_2}},\cdots, \widehat{a_{k_v}}, \cdots, Q^{(v)}) \neq 0$ for some  $k_1,\cdots,k_v \in \{1,\cdots,n\}$.  Without loss of generality, we eliminate $a_1^v$. So we have $g(a_{v+1},\cdots,a_n,Q^{(v)}) \neq 0$. Then there exist $x_1^v \in Q$ such that $0 \neq g(a_{v+1},\cdots,a_n,x_1^v) \in Q$. By induction hypothesis, we conclude that  $0 \neq g(f(a_1,x_1,0^{(m-2)}),\cdots,f(a_v,x_v,0^{(m-2)}),a_{v+1}^n) \subseteq Q$. By the hypothesis, we get  $g(s,f(a_i,x_i,0^{(m-2)}),1_A^{(n-2)}) \in Q$    for some $i \in \{1,\cdots,v\}$ or    $g(s,a_j,1_A^{(n-2)}) \in Q$ for some $j \in \{v+1,\cdots,n\}$. This implies that $g(s,a_i,1_A^{(n-2)}) \in Q$  for some $i \in \{1,\cdots,n\}$ which is a contradiction. Hence  we conclude that $g(a_1,\cdots,\widehat{a_{k_1}}, \cdots,\widehat{a_{k_2}},\cdots, \widehat{a_{k_u}}, \cdots, Q^{(v)})=0$ for each $k_1,\cdots,k_u \in \{1,\cdots,n\}$. 
\end{proof}
\begin{theorem} \label{14}
Let $S \subseteq A$ be an $n$-ary multiplicative set. If $Q$ is a strongly weakly $n$-ary $S$-prime hyperideal of $A$ that is not $n$-ary $S$-prime, then $g(Q^{(n)})=0$. 
\end{theorem}
\begin{proof}
Let $Q$ be a strongly weakly $n$-ary $S$-prime hyperideal of $A$ and $Q$ is associated to $s$. Suppose on the contrary that $0 \neq g(Q^{(n)})$.  We show that $Q$ is an $n$-ary $S$-prime hyperideal of $A$. Let $g(a_1^n) \in Q$ for $a_1^n \in A$. If $0 \neq g(a_1^n) \in Q$, then we have $g(s,a_i,1_A^{(n-2)}) \in Q$ for some $i \in \{1,\cdots,n\}$. Assume that $g(a_1^n)=0$.  From $ 0 \neq g(Q^{(n)})$, it follows that there exist $x_1^n \in Q$ such that $g(x_1^n) \neq 0$. By Theorem \ref{jalili},  we get the result that $0 \neq g(f(a_1,x_1,0^{(m-2)}), \cdots,f(a_n,x_n,0^{(m-2})) \subseteq Q$. By the hypothesis, we get $f(g(s,a_i,1_A^{(n-2)}),g(s,x_i,1_A^{(n-2)}),0^{(m-2)})=g(s,f(a_i,x_i,0^{(m-2)}) \subseteq Q$ for some $i \in \{1,\cdots,n\}$ which means $g(s,a_i,1_A^{(n-2)}) \in Q$ as $g(s,x_i,1_A^{(n-2)}) \in Q$. Thus $Q$ is an $n$-ary $S$-prime hyperideal of $A$, a contradiction. Consequently, $g(Q^{(n)})=0$.
\end{proof}
In view of the previous theorem   and Theorem 3.9 in \cite{mah6}, we have the following result.
\begin{corollary}\label{15}
Let $S \subseteq A$ be an $n$-ary multiplicative set and $Q$ be a strongly weakly $n$-ary $S$-prime hyperideal of $A$. Then $Q \subseteq rad(0)$ or $g(s,rad(0),1^{(n-2)}) \subseteq Q$ for some $s \in S$. 
\end{corollary}
The following result is a version of Theorem 4.4 in \cite{davvazz} where $k=1$. 
\begin{corollary}\label{16}
Let $Q$ be  a strongly weakly $n$-ary prime hyperideal of $A$ but is not $n$-ary prime hyperideal. Then $g(Q^{(n)})=0$.
\end{corollary}
\begin{proof}
By taking $S=\{1\}$ in Theorem \ref{14}, we are done.
\end{proof}
Suppose that $M $ is a non-empty set and $(A,f,g)$ is  a commutative Krasner $(m,n)$-hyperring .  If  $(M, f^{\prime})$ is an $m$-ary hypergroup and the map 
\[g^{\prime}:\underbrace{A \times ... \times A}_{n-1} \times M\longrightarrow 
P^*(M)\]
statisfied the following conditions:
\begin{itemize} 
\item[\rm{(1)}]~ 
$g^{\prime}(x_1^{n-1},f^{\prime}(a_1^m))=f^{\prime}(g^{\prime}(x_1^{n-1}
,a_1),...,g^{\prime}(a_1^{n-1}
,a_m))$
\item[\rm{(2)}]~ $g^{\prime}(x_1^{i-1},f(y_1^m),x_{i+1}^{n-1},a)=f^{\prime}(g^{\prime}(x_1^{i-1}
,y_1,x_{i+1}^{n-1},a),...,g^{\prime}(x_1^{i-1}
y_m,x_{i+1}^{n-1},a))$
\item[\rm{(3)}]~ $g^{\prime}(x_1^{i-1},g(x_i^{i+n-1}),x_{i+m}^{n+m-2},a)=
g^{\prime}(x_1^{n-1},g^{\prime}(x_m^{n+m-2},a))$
\item[\rm{(4)}]~$ 0=g^{\prime}(x_1^{i-1},0,x_{i+1}^{n-1},a)$,
\end{itemize} 
then the triple $(M, f^{\prime}, g^{\prime})$ is called an $(m, n)$-hypermodule over $(A,f,g)$ \cite{Anvariyeh}. In following theorem, Nakayama$^,$s Lemma is applied to a hyperideal that is strongly weakly $n$-ary $S$-prime.
\begin{theorem} \label{naka}
Let $(M,f^{\prime},g^{\prime})$ be an $(m,n)$-hypermodule over $(A,f,g)$ and $Q$ be a strongly weakly $n$-ary $S$-prime hyperideal of $A$ that is not $n$-ary $S$-prime. If $M=g^{\prime}(Q,1_A^{(n-2)},M)$, then $M=\{0\}$. 
\end{theorem}
\begin{proof}
Assume that  $Q$ is  a strongly weakly $n$-ary $S$-prime hyperideal of $A$ but is not $n$-ary $S$-prime and $M=g^{\prime}(Q,1_A^{(n-2)},M)$. By Theorem \ref{14}, we get the result that $g^{\prime}(g(Q^{(n)}),1_A^{(n-2)},M)=\{0\}$. Moreover, we have

 $\hspace{1cm}g^{\prime}(g(Q^{(n)}),1_A^{(n-2)},M)=g^{\prime}(g(Q^{(n-1)},1_A),1_A^{(n-2)},g^{\prime}(Q,1_A^{(n-2)},M))$
 
 $\hspace{4.4cm}=g^{\prime}(g(Q^{(n-1)},1_A),1_A^{(n-2)},M)$
 
 $\hspace{4.4cm}=g^{\prime}(g(Q^{(n-2)},1_A^{(2)}),1_A^{(n-2)},g^{\prime}(Q,1_A^{(n-2)},M))$
 
 $\hspace{4.4cm}=g^{\prime}(g(Q^{(n-2)},1_A^{(2)}),1_A^{(n-2)},M)$
 
 $\hspace{4.4cm}=\cdots$
 
 $\hspace{4.4cm}=g^{\prime}(Q,1_A^{(n-2)},g^{\prime}(Q,1_A^{(n-2)},M))$
 
 $\hspace{4.4cm}=g(Q,1_A^{(n-2)},M)$
 
 $\hspace{4.4cm}=M$. \\Then we conclude that  $M=\{0\}$. 
\end{proof}
\begin{theorem}\label{17}
Let $Q$ be a hyperideal of $A$ and $S \subseteq A$ is an $n$-ary multiplicative set  with $Q \cap S=\varnothing$. Then $Q$ is a strongly weakly $n$-ary $S$-prime hyperideal of $A$ if and only if 
there exists an element $s \in S$ such that for every $a \notin (Q : s)$, either $(Q : a) \subseteq (Q : s)$ or $(Q :a) =(0 : a)$.
\end{theorem}
\begin{proof}
($\Longrightarrow$)  Assume that $Q$ is a strongly weakly $n$-ary $S$-prime hyperideal of $A$ such that $Q$ is associated to $s$. Let $a \notin (Q : s)$ and $(Q :a)  \neq (0 : a)$. Then there exists $x \in (Q :a)$ such that $g(x,a,1_A^{(n-2)}) \neq 0$ as $(0 : a) \subseteq (Q :a)$. Since $0 \neq g(x,a,1_A^{(n-2)}) \in Q$ and $g(s,a,1_A^{(n-2)}) \notin Q$, we get $g(s,x,1_A^{(n-2)}) \in Q$. Take any $b \in (Q :a)$. So $g(a,b,1_A^{(n-2)}) \in Q$. Let $0 \neq g(a,b,1_A^{(n-2)})$. Therefore $g(s,b,1_A^{(n-2)}) \in Q$ which means $b \in (Q : s)$. If $0 = g(a,b,1_A^{(n-2)})$, then $0 \neq g(a,x,1_A^{(n-2)})=f(g(a,x,1_A^{(n-2)}),g(a,b,1_A^{(n-2)}),0^{(m-2)})=g(a,f(x,b,0^{(m-2)}),1_A^{(n-2)}) \in Q$. Since $Q$ is a strongly weakly $n$-ary $S$-prime hyperideal of $A$ and $a \notin (Q : s)$, we get the result that $f(g(s,x,1_A^{(n-2)}),g(s,b,1_A^{(n-2)}),0^{(m-2)})=g(s,f(x,b,0^{(m-2)}),1_A^{(n-2)}) \subseteq Q$ which implies $g(s,b,1_A^{(n-2)}) \in Q$ as $g(s,x,1_A^{(n-2)}) \in Q$. This means that $b \in (Q :s)$ and so $(Q : a) \subseteq (Q : s)$.\\

($\Longleftarrow$) Let $0 \neq g(Q_1^n) \subseteq Q$ for hyperideals $Q_1^n $ of $A$ such that $g(s,Q_i,1_A^{(n-2)}) \nsubseteq Q$ for  all $i \in \{1,\cdots,n\}$ and the element $s$ mentioned in the hypothesis. Take any  $a_i \in Q_i \backslash (Q : s)$ for $i \in \{1,\cdots,n\}$. Then $g(Q_1^{i-1},a_i,Q_{i+1}^n) \subseteq Q$ which means $g(Q_1^{i-1},1_A,Q_{i+1}^n) \subseteq (Q : a_i)$. Since $g(Q_1^{i-1},1_A,Q_{i+1}^n) \subseteq Q_j \nsubseteq (Q : s)$ for each $j \neq i$, we conclude that $ g(Q_1^{i-1},1_A,Q_{i+1}^n) \subseteq (Q : a_i) = (0 : a_i)$ which implies $g(Q_1^{i-1},a_i,Q_{i+1}^n)=0$. Now, let $a_i \in Q_i \cap  (Q : s)$ for some $i \in \{1,\cdots,n\}$ and $a_j \in Q_j$ for $j \in \{1,\cdots,\widehat{i},\cdots,n\}$. If $a_j \notin (Q : s)$, then we obtain $g(Q_1^{j-1},a_j,Q_{j+1}^n) =0$. Let $a_j \in (Q : s)$. By the hypothesis, there exists $x_j \in Q_j$ such that $g(s,x_j,1_A^{(n-2)}) \notin Q$. This means $x_j \notin (Q : s)$ and so $f(a_j,x_j,0^{(m-2)}) \nsubseteq (Q : s)$. 
Therefore  $g(f(a_1,x_1,0^{(m-2)}),\cdots, f(a_{i-1},x_{i-1},0^{(m-2)}),a_i,f(a_{i+1},x_{i+1},0^{(m-2)}),\cdots,f(a_n,x_n,\break 0^{(m-2)}))=0$. Then we conclude that $g(a_1^n)=0$ and so $g(Q_1^n)=0$ which is impossible. Thus $Q$ is a strongly weakly $n$-ary $S$-prime hyperideal of $A$.
\end{proof}
\begin{corollary}
Let $Q$ be a hyperideal of $A$. Then $Q$ is a strongly weakly $n$-ary prime hyperideal of $A$ if and only if  for every $a \notin Q$, either $(Q : a) = Q$ or $(Q :a) =(0 : a)$.
\end{corollary}
\begin{proof}
By taking $S = \{1\}$ in Theorem \ref{17}, we are done.
\end{proof}
\begin{theorem} \label{18}
Let $Q$ be a hyperideal of $A$ and $S \subseteq A$  is an $n$-ary multiplicative set. If  $Q$ is a strongly weakly $n$-ary $S$-prime hyperideal of $A$ that is not $n$-ary $S$-prime, then $g(g(s,rad(0),1_A^{(n-2)}),Q^{(n-1)})=0$ for some $s \in S$.
\end{theorem}
\begin{proof}
Suppose that $Q$ is a strongly weakly $n$-ary $S$-prime hyperideal of $A$. Then there exists an element $s \in S$ such that for every $a \notin (Q : s)$, either $(Q : a) \subseteq (Q : s)$ or $(Q :a) =(0 : a)$, by Theorem \ref{17}. Take any $x \in rad(0)$. If $x \in (Q : s)$, then $g(s,x,1_A^{(n-2)}) \in Q$ and so $g(g(s,rad(0),1_A^{(n-2)}),Q^{(n-1)})=0$ by Theorem \ref{14}. Now, let $x \notin (Q : s)$. This implies that $(Q : x) \subseteq (Q : s)$ or $(Q :x) =(0 : x)$. The first case leads to the following contradiction. Since $x \in rad(0)$,  there exists $t \in \mathbb{N}$ such that $g(x^{(t)},1_A^{(n-t)})=0$ for $t \leq n$ or $g_{(l)}(x^{(t)})=0$ for $t=l(n-1)+1$. Assume that $t$ is a minimal integer  satisfying the possibilities. If $g(x^{(t)},1_A^{(n-t)})=0$  for $t \leq n$, then $g(x^{(t-1)},1_A^{(n-t+1)}) \in (Q : x) \subseteq (Q : s)$ which means $g(g(x^{(t-1)},1_A^{(n-t+1)}),s,1_A^{(n-2)})=g(x^{(t-1)},s,1_A^{(n-t)}) \in Q$. If $0 \neq g(x^{(t-1)},s,1_A^{(n-t)})$, then we get $g(s,x,1_A^{(n-2)}) \in Q$ which is impossible. Therefore $g(x^{(t-1)},s,1_A^{(n-t)})=0$. Assume that $u$ is a minimal integer satisfying $g(s,x^u,1_A^{(n-2)})=0$. Assume that $g(g(s,x,1_A^{(n-2)}),q_1^{n-1}) \neq 0$ for some $q_1^{n-1} \in Q$. Since $Q$ is a strongly weakly $n$-ary $S$-prime hyperideal of $A$ and $0 \neq g(g(s,x,1_A^{(n-2)}),q_1^{n-1})=g(g(s,x,1_A^{(n-2)}),f(x^{u-1},g(q_1^{n-1},1_A),0^{(m-2)}),1_A^{(n-2)}) \subseteq Q$, we obtain $g(s,f(x^{u-1},g(q_1^{n-1},1_A),0^{(m-2)}),1_A^{(n-2)}) \subseteq Q$. Since $g(s,q_1^{n-1}) \in Q$, we have  $0 \neq g(s,x^{(u-1)},1_A^{(n-u)}) \in Q$ which implies $g(s,x,1_A^{(n-2)}) \in Q$ which is a contradiction. If $g_{(l)}(x^{(t)})=0$ for $t=l(n-1)+1$, then by using a similar argument, we get a contradiction. In the second case, we get the result that $g(x,Q,1_A^{(n-2)})=0$ as $Q \subseteq (Q : x)$. Thus $g(s,g(x,Q,1_A^{(n-2)}),Q^{(n-2)})=g(g(s,x,1_A^{(n-2)}),Q^{(n-1)})=0$ and so  $g(g(s,rad(0),1_A^{(n-2)}),Q^{(n-1)})=0$.
\end{proof}
\begin{corollary} \label{19}
Let $Q_1$ and $Q_2$ be  two hyperideals of $A$ and $S \subseteq A$  is an $n$-ary multiplicative set. If  $Q_1$ and $Q_2$ are  strongly weakly $n$-ary $S$-prime hyperideals of $A$ that are not $n$-ary $S$-prime, then we have $g(g(s,Q_1,1_A^{(n-2)}),Q_2^{(n-1)})=g(g(s,Q_2,1_A^{(n-2)}),Q_1^{(n-1)})=0$ for some $s \in S$.
\end{corollary}
\begin{proof}
Let $Q_1$ and $Q_2$ be  two strongly weakly $n$-ary $S$-prime hyperideals of $A$ that are not $n$-ary $S$-prime. Then we get the result that  $Q_1, Q_2 \subseteq rad(0)$ by Theorem \ref{14}. By Theorem \ref{18}, we have $g(g(s,Q_1,1_A^{(n-2)}),Q_2^{(n-1)})=g(Q_1,g(s,Q_2^{(n-1)}),1_A^{(n-2)}) \subseteq g(rad(0), g(s,Q_2^{(n-1)}),1_A^{(n-2)})=g(g(s,rad(0),1_A^{(n-2)}),Q_2^{(n-1)})=0 $ . Similarly, we can conclude that $g(g(s,Q_2,1_A^{(n-2)}),Q_1^{(n-1)})=0$
\end{proof}
\begin{corollary}\label{120}
Let $Q$ be  a strongly weakly $n$-ary prime hyperideal of $A$ but is not $n$-ary prime hyperideal. Then $g(rad(0),Q^{(n-1)})=0$.
\end{corollary}
\begin{proof}
By taking $S=\{1\}$ in Theorem \ref{18}, we are done.
\end{proof}
The notion of Krasner $(m,n)$-hyperring of fractions was introduced and studied in \cite{mah5}.
\begin{theorem} \label{121}
Let $S \subseteq A$ be an $n$-ary multiplicative set with $1_A \
\in S$. If $Q$ is a weakly $n$-ary $S$-prime hyperideal of $A$, then $S^{-1}Q$ is a weakly  $n$-ary prime hyperideal of $S^{-1}A$.
\end{theorem}
\begin{proof}
Let $0 \neq G(\frac{a_1}{s_1},...,\frac{a_n}{s_n}) \in S^{-1}Q$ for $\frac{a_1}{s_1},...,\frac{a_n}{s_n} \in S^{-1}A$.
This means that $\frac{g(a_1^n)}{g(s_1^n)} \in S^{-1}Q$ and so there exists $t \in S$ such that $ g(t,g(a_1^n),1_A^{(n-2)}) \in Q$. Since $Q$ is a weakly $n$-ary $S$-prime hyperideal of $A$ and $0 \neq g(g(t,a_1,1^{(n-2)}),a_2^n) \in Q$, there exists $s \in S$ such that $g(s,g(t,a_1,1_A^{(n-2)}),1_A^{(n-2)})=g(s,t,a_1,1_A^{(n-3)}) \in Q$ or $g(s,a_i,1_A^{(n-2)}) \in Q$ for some $i \in \{2,\cdots,n\}$. Therefore we obtain $G(\frac{a_1}{s_1},\frac{1_A}{1_A}^{(n-1)})=\frac{g(a_1,1_A^{(n-1)})}{g(s_1,1_A^{(n-1)})}=\frac{g(s,t,a_1,1_A^{(n-3)})}{g(s,t,s_1,1_A^{(n-2)})} \in S^{-1}Q$ or $G(\frac{a_i}{s_i},\frac{1_A}{1_A}^{(n-1)})=\frac{g(a_i,1_A^{(n-1)})}{g(s_i,1_A^{(n-1)}))}=\frac{g(s,a_i,1_A^{(n-2)})}{g(s,s_i,1_A^{(n-2)})} \in S^{-1}Q$  for some $i \in \{2,\cdots,n\}$. Hence  $S^{-1}Q$ is a weakly  $n$-ary prime hyperideal of $S^{-1}A$.
\end{proof}
\begin{theorem}
Let $S \subseteq A$ be an $n$-ary multiplicative set with $1_A  \in S$ and $Q$ be a hyperideal of $A$ with $Q \cap S=\varnothing$ . If $S^{-1}Q$ is a weakly $n$-ary prime hyperideal of $S^{-1}A$ and $S^{-1}Q\cap A=(Q : s)$ for some $ s \in S$, then $Q$ is a weakly $n$-ary $S$-prime hyperideal of $A$
\end{theorem}
\begin{proof}
Let $S^{-1}Q$ be a weakly $n$-ary prime hyperideal of $S^{-1}A$ and $S^{-1}Q  \cap A=(Q : s)$ for some $ s \in S$. Suppose that $0 \neq g(a_1^n) \in Q$ for some $a_1^n \in A$. Then we have $0 \neq G(\frac{a_1}{1_A},\cdots,\frac{a_n}{1_A}) \in S^{-1}Q$. Since $S^{-1}Q$ is a weakly  $n$-ary prime hyperideal of $S^{-1}A$, we get $\frac{a_i}{1} \in S^{-1}Q$ for some $i \in \{1,\cdots,n\}$ which means $g(t,a_i,1_A^{(n-2)}) \in Q$ for some $t \in S$. Therefore $a_i=\frac{g(t,a_i,1_A^{(n-2)})}{g(t,1_A^{(n-1)})} \in S^{-1}Q$. This means $a_i \in (Q : s)$. Therefore we have $g(s,a_i,1_A^{(n-2)}) \in Q$. This shows that $I$ is a weaky  $n$-ary $S$-prime hyperideal of $A$.
\end{proof}
\begin{theorem}
Let $S \subseteq A$ be an $n$-ary multiplicative set with $1_A \in S$ and $Q$ be a hyperideal of $A$ with $Q \cap S=\varnothing$. If   there exists $t \in S$ satisfying $(Q : s) \subseteq (Q : t)$ for all $s \in S$ and $S^{-1}Q$ is a weakly $n$-ary prime hyperideal of $S^{-1}A$, then $Q$ is a weakly $n$-ary $S$-prime hyperideal of $A$.
\end{theorem}
\begin{proof}
Assume that there exists $t \in S$ satisfying $(Q : s) \subseteq (Q : t)$ for all $s \in S$. Let $0 \neq g(a_1^n) \in Q$ for $a_1^n \in A$. Therefore we get $0 \neq G(\frac{a_1}{1_A},\cdots,\frac{a_n}{1_A}) \in S^{-1}Q$. It follows that $\frac{a_i}{1} \in S^{-1}Q$ for some $i \in \{1,\cdots,n\}$ as $S^{-1}Q$ is a weakly $n$-ary prime hyperideal of $S^{-1}A$. Hence $g(s,a_i,1_A^{(n-2)}) \in Q$ for some $s\in S$ which implies $a_i \in (Q : s) \subseteq (Q : t)$ and so $g(t,a_i,1_A^{(n-2)}) \in Q$. Consequently, $Q$ is a weakly $n$-ary $S$-prime hyperideal of $A$.
\end{proof}
\begin{theorem}
Let $S \subseteq A$ be an $n$-ary multiplicative set with $1_A  \in S$. If   $A$ is an $n$-ary hyperintegral
domain and $S^{-1}A$ is a hyperfield, then $\langle 0 \rangle$ is the only weakly $n$-ary $S$-prime hyperideal of $A$.
\end{theorem}
\begin{proof}
Assume that $Q \neq 0$ is a  weakly $n$-ary $S$-prime hyperideal of $A$. Take $a \in Q \backslash \{0\}$. Then there exists $x \in A \backslash \{0\}$ and $s \in S$ satisfying $G(\frac{a}{1_A},\frac{x}{s},\frac{1_A}{1_A}^{(n-2)})=\frac{g(a,x,1_A^{(n-2)})}{g(s,1_A^{(n-1)})}=\frac{1_A}{1_A}$ as $S^{-1}A$ is a hyperfield. Then we conclude that there exists $t \in S$ such that    $0 \in g(t,f(g(a,x,1_A^{(n-2)}),-g(s,1_A^{(n-1)}),0^{(m-2)}),1_A^{(n-2)})=f(g(t,a,x,1_A^{(n-3)}),-g(t,s,1_A^{(n-2)}),0^{(m-2)})$. Therefore we get  $g(t,a,x,1_A^{(n-3)}) \in f(g(t,s,1_A^{(n-2)}),0^{(m-1)}) \subseteq S$. Since $0 \neq g(t,a,x,1_A^{(n-3)}) \in Q$, we get $Q \cap S \neq \varnothing$ which is impossible. Thus $\langle 0 \rangle$ is the only weakly $n$-ary $S$-prime hyperideal of $A$.
\end{proof}
Recall from \cite{d1} that a mapping $h : A_1 \longrightarrow A_2$ is called a homomorphism where $(A_1, f_1, g_1)$ and $(A_2, f_2, g_2)$ are commutative Krasner $(m, n)$-hyperrings if for all $a^m _1, b^n_ 1 \in A_1$ we have
\begin{itemize}
\item[\rm{(i)}]~$h(f_1(a_1,\cdots, a_m)) = f_2(h(a_1),\cdots,h(a_m)),$
\item[\rm{(ii)}]~$h(g_1(b_1,\cdots, b_n)) = g_2(h(b_1),\cdots,h(b_n))$
\item[\rm{(iii)}]~$h(1_{A_1})=1_{A_2}.$
\end{itemize}
\begin{theorem} \label{122}
Let $(A_1,f_1,g_1)$ and $ (A_2,f_2,g_2)$ be two commutative Krasner $(m,n)$-hyperrings, $h:A_1 \longrightarrow A_2$  a monomorphism and $S \subseteq A_1$  an $n$-ary multiplicative set. If  $Q_2$ is a weakly  $n$-ary $h(S)$-prime hyperideal of $A_2$, then $h^{-1}(Q_2)$ is a weakly $n$-ary $S$-prime hyperideal of $A_1$. 
\end{theorem}
\begin{proof}
 Assume that $Q_2$ is a weakly $n$-ary $h(S)$-prime hyperideal of $A_2$. Then there exists $s \in S$ such that for all $b_1^n \in A_2$ with $0 \neq g_2(b_1^n) \in Q_2$, we have $g_2(h(s),b_i,1_{A_2}^{(n-2)}) \in Q_2$ for some $i \in \{1,\cdots,n\}$. Put $Q_1=h^{-1}(Q_2)$. It is easy to see that $Q_1 \cap S = \varnothing$. Let $0 \neq g_1(a_1^n) \in Q_1$ for $a_1^n \in A_1$. Then $0 \neq h(g_1(a_1^n))=g_2(h(a_1),...,h(a_n)) \in Q_2$ as $h$ is a monomorphism. So, we have $g_2(h(s),h(a_i),1_{A_2}^{(n-2)})=h(g_1(s,a_i,1_{A_1}^{(n-2)})) \in Q_2$ for some $i \in \{1,\cdots,n\}$ which implies $g_1(s,a_i,1_{A_1}^{(n-2)}) \in h^{-1}(Q_2)=Q_1$. Consequently, $h^{-1}(Q_2)$ is a weakly $n$-ary $S$-prime hyperideal of $A_1$.
\end{proof}
\begin{corollary}
Let $S \subseteq A_1$  be an $n$-ary multiplicative set. If $A_1$ is a subhyperring of $A_2$ and $Q_2$ is a weakly $n$-ary $S$-prime hyperideal of $A_2$, then $Q_2 \cap A_1$ is a weakly $n$-ary $S$-prime hyperideal of $A_1$.
\end{corollary}
\begin{proof}
Consider the monomorphism $h : A_1 \longrightarrow A_2$ , defined by $h(a)=a$. Since $h^{-1}(Q_2)=Q_2 \cap A_1$, we conclude that $Q_2 \cap A_1$ is a weakly $n$-ary $S$-prime hyperideal of $A_1$, by Theorem \ref{122}.
\end{proof}
Assume that $(A_1, f_1, g_1)$ and $(A_2, f_2, g_2)$ are two commutative Krasner $(m,n)$-hyperrings such that $1_{A_1}$ and $1_{A_2}$ are two scalar identities of $A_1$ and $A_2$, respectively. Then the $(m, n)$-hyperring $(A_1 \times A_2, f ,g )$ is defined by $m$-ary hyperoperation
$f $ and $n$-ary operation $g$, as follows:

$\hspace{1cm} f((x_{1}, y_{1}),\cdots,(x_m,y_m)) = \{(x,y) \ \vert \ \ x \in f_1(x_1^m), y \in f_2(y_1^m) \}$

$\hspace{1cm} g ((a_1,b_1),\cdots,(a_n,b_n)) =(g_1(a_1^n),g_2(b_1^n)) $,\\
for all $x_1^m,a_1^n \in A_1$ and $y_1^m,b_1^n \in A_2$ \cite{mah2}. 
\begin{theorem} \label{cart}
Assume that $(A_1, f_1, g_1)$ and $(A_2, f_2, g_2)$ are two commutative Krasner $(m,n)$-hyperrings such that $1_{A_1}$ and $1_{A_2}$ are two scalar identities of $A_1$ and $A_2$, respectively. Suppose that $S_1 \subseteq A_1$ and $S_2 \subseteq A_2$ are $n$-ary multiplicative sets and, $Q_1$ and $Q_2$ are nonzero hyperideals of $A_1$ and $A_2$, respectively. Then the following assertions are equivalent:
 \begin{itemize} 
\item[\rm{(i)}]~ $Q_1 \times Q_2$ is a weakly $n$-ary $S_1 \times S_2$-prime hyperideal of $A_1 \times A_2$.
\item[\rm{(ii)}]~ $Q_1$ is an $n$-ary $S_1$-prime hyperideal of $A_1$  and $Q_2 \cap S_2 \neq \varnothing$ or $Q_2$ is an $n$-ary $S_2$-prime hyperideal of $A_2$  and $Q_1 \cap S_1  \neq \varnothing$
\item[\rm{(iii)}]~ $Q_1 \times Q_2$ is an $n$-ary $S_1 \times S_2$-prime hyperideal of $A_1 \times A_2$.
\end{itemize} 
\end{theorem}
\begin{proof}
(i) $\Longrightarrow$ (ii) Assume that $(0,0) \neq (a,b) \in Q_1 \times Q_2$ for some $a \in Q_1$ and $b \in Q_2$. So, we have $(0,0) \neq (a,b)=g((a,1_{A_2}),(1_{A_1},1_{A_2})^{(n-2)},(1_{A_1},b))\in Q_1 \times Q_2$. Then there exist $s_1 \in S_1$ and $s_2 \in S_2$ such that  $g((s_1,s_2),(a,1_{A_2}),(1_{A_1},1_{A_2})^{(n-2)})=(g_1(s_1,a,1_{A_1}^{(n-2)}),g_2(s_2,1_{A_2}^{(n-1)})) \in Q_1 \times Q_2$ or $g((s_1,s_2),(1_{A_1},b),(1_{A_1},1_{A_2})^{(n-2)})=(g_1(s_1,1_{A_1}^{(n-1)}),
g_2(s_2,b,1_{A_2}^{(n-2)})) \in Q_1 \times Q_2$ as $Q_1 \times Q_2$ is a weakly $n$-ary $S_1 \times S_2$-prime hyperideal of $A_1 \times A_2$. This implies that $Q_1 \cap S_1  \neq \varnothing$ or $Q_2 \cap S_2 \neq \varnothing$. Consider $Q_2 \cap S_2 \neq \varnothing$. It follows that $Q_1 \cap S_1 = \varnothing$ because $(Q_1 \times Q_2) \cap (S_1 \times S_2)= \varnothing$. Now, let $g(a_1^n) \in Q_1$ for $a_1^n \in A_1$. From $Q_2 \cap S_2 \neq \varnothing$ it follows that there exists $0 \neq r \in Q_2 \cap S_2$. Then we have $0 \neq (g_1(a_1^n),g_2(r,1_{A_2}^{(n-1)}))=g((a_1,r),(a_2,1_{A_2}),\cdots(a_n,1_{A_2})) \in Q_1 \times Q_2$. Therefore we get the result that $(g_1(s_1,a_1,1_{A_1}^{(n-2)}),g_2(s_2,r,1_{A_2}^{(n-2)}))=g((s_1,s_2),(a_1,r),(1_{A_1},1_{A_2})^{(n-2)}) \in Q_1 \times Q_2$ or $(g_1(s_1,a_i,1_{A_1}^{(n-2)}),g_2(s_2,1_{A_2}^{(n-1)}))=g((s_1,s_2),(a_i,1_{A_2}),(1_{A_1},1_{A_2})^{(n-2)})\in Q_1 \times Q_2$ for some $ i \in \{2,\cdots,n\}$. Thus $g_1(s_1,a_i,1_{A_1}) \in Q_1$ which shows 
$Q_1$ is an $n$-ary $S_1$-prime hyperideal of $A_1$.\\

(ii) $\Longrightarrow$ (iii) Let $Q_1$ be an $n$-ary $S_1$-prime hyperideal of $A_1$  and $Q_2 \cap S_2  \neq \varnothing$. So, there exists $s_2 \in Q_2 \cap S_2$. Let $g((a_1,b_1),\cdots,(a_n,b_n))=(g_1(a_1^n),g_2(b_1^n)) \in Q_1 \times Q_2$ for some $a_1^n \in A_1$ and $b_1^n \in A_2$. Since $Q_1$ is an $n$-ary $S_1$-prime hyperideal of $A_1$ and $g_1(a_1^n) \in Q_1$, there exists $s_1 \in S_1$ such that $g_1(s_1,a_i,1_{A_1}^{(n-2)}) \in Q_1$ for some $i \in \{1,\cdots,n\}$. This implies that $g((s_1,s_2),(a_i,b_i),(1_{A_1},1_{A_2})^{(n-2)})=(g_1(s_1,a_i,1_{A_2}^{(n-2)}),g_2(s_2,b_i,1_{A_1}^{(n-2)})) \in Q_1 \times Q_2$. Thus $Q_1 \times Q_2$ is an $n$-ary $S_1 \times S_2$-prime hyperideal of $A_1 \times A_2$. Similarly, one can prove that $Q_1 \times Q_2$ is an $n$-ary $S_1 \times S_2$-prime hyperideal of $A_1 \times A_2$ if $Q_2$ is an $n$-ary $S_2$-prime hyperideal of $A_2$  and $Q_1 \cap S_1 \neq \varnothing$.\\

(iii) $\Longrightarrow$ (i) It is obvious.
\end{proof}
\begin{corollary}
Let  $(A_1, f_1, g_1), \cdots, (A_u, f_u, g_u)$ be  commutative Krasner $(m,n)$-hyperrings,   $S_1 \subseteq A_1,\cdots, S_u \subseteq A_u$  $n$-ary multiplicative sets and  $Q_1^u$   hyperideals of $A_1^u$, respectively. Then $Q_1 \times \cdots, \times Q_u$ is a weakly $n$-ary $(S_1 \times \cdots, \times S_u)$-prime hyperideal of $A_1 \times \cdots, \times A_u$ if and only if  $Q_i$ is an $n$-ary $S_i$-prime hyperideal of $A_i$  for some $ i \in \{1,\cdots,u\}$ and $Q_j \cap S_j \neq \varnothing$ for all $j \neq i$
\end{corollary}
\begin{proof}
We use the induction on $u$. The claim is clear for $u=1$. If $u=2$, then we are done by Theorem \ref{cart}. Let the claim be true for all $i  < u$.  By applying Theorem \ref{cart} to $Q=Q^{\prime} \times Q_u$ where $Q^{\prime}=Q_1 \times \cdots, \times Q_{u-1}$, we are done.
\end{proof}
\begin{theorem}
Let $(A_1, f_1, g_1)$ and $(A_2, f_2, g_2)$ be two commutative Krasner $(m,n)$-hyperrings, $1_{A_1}$ and $1_{A_2}$  scalar identities of $A_1$ and $A_2$, respectively, and $S_1 \subseteq A_1$ and $S_2 \subseteq A_2$  $n$-ary multiplicative sets. If $A_1$ and $A_2$ are hyperfields, then every proper hyperideal of $A_1 \times A_2$ is weakly $S$-prime.
\end{theorem}
\begin{proof}
Let $A_1$ and $A_2$ be hyperfields. Then $A_1 \times A_2$ has exactly three proper hyperideals $\{0\} \times A_2$, $A_1 \times \{0\}$ and $\{0\} \times \{0\}$. Thus, we conclude that these hyperideals are weakly $S$-prime by Theorem \ref{cart}.
\end{proof}


\end{document}